\documentclass[12pt, a4paper, twoside]{amsart}

\usepackage[margin=2.85cm]{geometry}

\usepackage[T1]{fontenc}
\usepackage[latin1]{inputenc}
\usepackage{amsmath}
\usepackage{amsthm}
\usepackage{amsfonts}
\usepackage[UKenglish]{babel}
\usepackage{enumerate}
\usepackage[dvips]{graphicx}
\usepackage{bm}
\usepackage{ae}

\theoremstyle{definition}

\theoremstyle{remark}

\newtheorem{Rem}{Remark}
\theoremstyle{plain}
\newtheorem{thm}{Theorem}[section]
\newtheorem{lem}[thm]{Lemma}

\newcommand{\abs}[1]{\ensuremath{\left\vert #1 \right\vert}}

\DeclareMathOperator{\mat}{Mat}
\DeclareMathOperator{\Span}{Span}

\newcommand{\eps}{\varepsilon}
\newcommand{\bfx}{{\bf x}}
\newcommand{\bfy}{{\bf y}}
\newcommand{\bfv}{{\bf v}}
\newcommand{\bfu}{{\bf u}}
\newcommand{\bfq}{{\bf q}}
\newcommand{\bfe}{{\bf e}}
\newcommand{\bfze}{{\bf 0}}

\newcommand{\bfqu}{{\bf q}^{(1)}}
\newcommand{\bfqj}{{\bf q}^{(j)}}
\newcommand{\bfqd}{{\bf q}^{(d )}}
\newcommand{\bfqdpu}{{\bf q}^{(d+1)}}
\newcommand{\bfqm}{{\bf q}^{(m)}}
\newcommand{\bfqtiu}{{\bf \tilde q}^{(1)}}
\newcommand{\bfqtij}{{\bf \tilde q}^{(j)}}
\newcommand{\bfqtid}{{\bf \tilde q}^{(d)}}
\newcommand{\und}{\{1,\ldots,d\}}
\newcommand{\unn}{\{1,\ldots,n\}}
\newcommand{\unm}{\{1,\ldots,m\}}
\newcommand{\eti}{\tilde \bfe}
\newcommand{\Om}{\Omega}
\newcommand{\Ati}{\tilde A}
\newcommand{\Ainf}{A_{\infty}}

\newcommand{\lam}{\lambda}
\newcommand{\calC}{{\mathcal C}}
\newcommand{\vol}{{\rm vol}}

\renewcommand{\Im}{{\rm Im} \, }
\newcommand{\tra}{ \, \, ^t  }
\newcommand{\R}{\mathbb R}
\newcommand{\Z}{\mathbb Z}
\newcommand{\Q}{\mathbb Q}
\newcommand{\N}{\mathbb N}
\newcommand{\Qbar}{\overline \Q}

\begin{document}

\title[A converse to linear independence criteria]{A converse to
  linear independence criteria, valid almost everywhere}

\author{S. Fischler}

\address{S. Fischler, Equipe d'Arithm\'etique et de G\'eom\'etrie Alg\'ebrique,
  Universit\'e Paris-Sud, B\^atiment 425, 91405 Orsay Cedex, France}

\email{stephane.fischler@math.u-psud.fr}

\author{M. HUSSAIN}

\address{M. Hussain, Department of Mathematics, Aarhus University, Ny Munkegade 118,
  DK-8000 Aarhus C, Denmark}


\email{mhuss@imf.au.dk}

\author{S. KRISTENSEN}

\address{S. Kristensen, Department of Mathematics, Aarhus University, Ny Munkegade 118,
  DK-8000 Aarhus C, Denmark}

\thanks{SK's research supported by the Danish Research Council for
  Independent Research, MH's research is partially supported  by the Danish Research Council for Independent Research,  and SF's partially  by Agence Nationale de la Recherche (project HAMOT, ref. ANR 2010 BLAN-0115-01).}

\email{sik@imf.au.dk}

\author{J. Levesley}

\address{J. Levesley, Department of Mathematics, University of York,
York, YO10 5DD}

\email{jl107@york.ac.uk}


\begin{abstract}
  We prove a weighted analogue of the Khintchine--Groshev Theorem,
  where the distance to the nearest integer is replaced by the
  absolute value. This is subsequently applied to proving the
  optimality of several linear independence criteria over the field of
  rational numbers.
\end{abstract}

\maketitle

\section{Introduction}
\label{sec:setup-result}

Let $n \geq 1$ and let $\psi_1, \dots, \psi_n : \mathbb{R}^+
\rightarrow \mathbb{R}^+$ be functions tending to zero. We will refer
to these functions as approximating functions or error functions. Let $\underline{\psi} =
(\psi_1, \dots, \psi_n)$. An $m \times n$-matrix
$X=(x_{ij})_{\substack{1 \leq i \leq n \\ 1 \leq j \leq m}} \in
\mathbb{R}^{mn}$ (or a system of linear forms) is said to be
$\underline{\psi}$-approximable if
\begin{equation}
  \label{eq:2}
  \abs{q_1 x_{1i} + \dots + q_m x_{mi}} < \psi_i(\abs{\mathbf{q}}), \quad 1
  \leq i \leq n
\end{equation}
for infinitely many integer vectors $\mathbf{q} = (q_1, \dots, q_m)
\in \mathbb{Z}^m \setminus \{\mathbf 0\}$. The norm $\abs{\mathbf{q}}$
is the supremum norm here and elsewhere.  We will denote the set of
$\underline{\psi}$-approximable linear forms inside the set
$[-\frac{1}{2}, \frac{1}{2}]^{mn}$ by $W_0(m,n,\underline{\psi})$.

The similarity between the $\underline{\psi}$-approximable linear
forms studied here and the simultaneously $\psi$-approximable linear forms usually
studied in Diophantine approximation is clear. However, in the
classical setup one studies the distance to the nearest integer rather
than the absolute value.

A major breakthrough in the classical theory was the
Khintchine--Groshev theorem \cite{groshev38, MR1544787}, which
establishes a zero-one law for the set of
$\underline{\psi}$-approximable matrices depending on the convergence
or divergence of a certain series. In the absolute value setting, an
analogue of this result was recently obtained by Hussain and Levesley
\cite{hussain}. Their result covers only the case $\psi_1 = \cdots =
\psi_n$ with this approximating function being monotonic.  The
condition of monotonicity was removed by Hussain and Kristensen
\cite{hussain:_metric} in the case of a single approximating function.

In the present paper, we extend the results of \cite{hussain} and
\cite{hussain:_metric} to the weighted setup, \emph{i.e.}, the case of
more than one approximating function. This has applications to linear
independence criteria, as we shall see below. Our zero-one law states
the following.

\begin{thm}
  \label{thm:KG}
  Let $m > n > 0$ and let $\psi_1, \dots, \psi_n$ be approximating
  functions as above. Then, if $(m,n) \neq (2,1)$,
  \begin{equation*}
    \lambda_{mn}(W_0(m,n,\underline{\psi})) =
    \begin{cases}
      0 & \text{if } \sum_{r=1}^\infty \psi_1(r) \cdots \psi_n(r)
      r^{m-n-1} < \infty \\
      1 & \text{if } \sum_{r=1}^\infty \psi_1(r) \cdots \psi_n(r)
      r^{m-n-1} = \infty,
    \end{cases}
  \end{equation*}
  where $\lambda_{mn}$ denotes the $m n$-dimensional Lebesgue
  measure. If $(m,n)=(2,1)$, the same conclusion holds provided the
  error function is monotonic.
\end{thm}

The case $m \leq n$ is of less interest in general, and of no
particular interest to us for applications. Briefly, in this case the
set $W_0(m,n,\underline{\psi})$ becomes a subset of a lower
dimensional set. An easy instance is that of $m=n=1$, where it is
straightforward to prove that the set is in fact a singleton -- see,
\emph{e.g.}, Lemma 1 in \cite{dickinson} for details. This is in contrast to the classical case, where approximation to the nearest integer is considered. Here, the result is independent of the relative sizes of $m$ and $n$.

This setting where linear forms are very small at some points appears
  in linear independence criteria. To begin with, let us consider
the case of one point. Siegel has proved, using essentially a
determinant argument, that the existence of $m$ linearly independent
linear forms very small at a given point $\bfe_1 =
(\xi_1,\ldots,\xi_m)\in\R^m$ implies a lower bound on the dimension of
the $\Q$-vector space spanned by $\xi_1,\ldots,\xi_m$. A precise
statement is given by Theorem \ref{thcrit} below with assumption $(i)$
and $n=1$; notice that $\dim_\Q\Span_\Q(\xi_1,\ldots,\xi_m)$ is equal
to the dimension of the smallest subspace $F$ of $\R^m$, defined over
the rationals, which contains the point $\bfe_1 =
(\xi_1,\ldots,\xi_m)$. The reader may refer to \S 8 of \cite{Boualg}
for classical facts about subspaces defined over the rationals, to
Lemma 1 of \cite{SFnestsev} (\S 2.3) for a generalization of this
equality, and to \cite{EMS} (especially pp.~81--82 and 215--216) for
more details on Siegel's criterion, including applications.

On the other hand, still in the case of one point $\bfe_1 =
(\xi_1,\ldots,\xi_m)$, Nesterenko has derived \cite{Nesterenkocritere}
a similar lower bound for $\dim_\Q\Span_\Q(\xi_1,\ldots,\xi_m)$ from
the existence of just {\em one} linear form (for each $Q $
sufficiently large), small at $\bfe_1$ but not too small: see Theorem
\ref{thcrit} below with assumption $(ii)$ and $n=1$. The most striking
application of his result is the proof by Rivoal \cite{RivoalCRAS} and
Ball-Rivoal \cite{BR} that infinitely many values of Riemann $\zeta$
function at odd integers $s\geq 3$ are irrational.

The first author has generalized recently Nesterenko's linear
independence criterion to linear forms small at several points (see
\cite{SFnestsev}, Theorem 3). The statement is the following, with
assumption $(ii)$. We provide also (under assumption $(i)$) the
analogue of Siegel's criterion in this setting (see \cite{SFnestsev},
\S 2.4, Proposition 1). We denote by $\cdot$ the canonical scalar
product on $\R^m$ (which allows us to consider a linear form as the
scalar product with a given vector), and by $o(1)$ any sequence that
tends to 0 as $Q\to\infty$.

\begin{thm} \label{thcrit} Let $m > n > 0$, and $ \bfe_1,\ldots,\bfe_n
  \in \R^m$. Let $\tau_1,\ldots,\tau_n$ be positive real numbers.
  Assume that one of the following holds:
   \begin{itemize}
   \item[$(i)$] The vectors $ \bfe_1,\ldots,\bfe_n $ are linearly
     independent, and for infinitely many integers $Q$ there exist $m$
     linearly independent vectors $\bfqu, \ldots, \bfqm \in\Z^m$ such
     that, for any $ j \in \unm$:
   $$| \bfqj | \leq Q \mbox{ and } | \bfqj  \cdot \bfe_i | \leq
   Q^{-\tau_i +o(1)} \mbox{ for any } i \in \unn. $$
 \item[$(ii)$] The numbers $\tau_1,\ldots,\tau_n$ are pairwise
   distinct, and for any sufficiently large integer $Q$ there exists
   $\bfq \in\Z^m$ such that
   $$| \bfq | \leq Q \mbox{ and } | \bfq  \cdot \bfe_i | = Q^{-\tau_i +o(1)}   \mbox{ for any } i \in \unn.   $$
     \end{itemize}
     Then we have
   $$\dim F \geq n+    \tau_1+\ldots+\tau_n$$
   for any subspace $F$ of $\R^m$ which contains $
   \bfe_1,\ldots,\bfe_n$ and is defined over the rationals.
   \end{thm}

   Note that $ \bfe_1,\ldots,\bfe_n $ are always $\R$-linearly
   independent: this is assumed in $(i)$, and it is an easy
   consequence of assumption $(ii)$ since $\tau_1,\ldots,\tau_n$ are
   pairwise distinct (see \cite{SFnestsev}, \S 3.2). The point is that
   $\Span_\R ( \bfe_1,\ldots,\bfe_n)$ is not defined over the
   rationals.

   The conclusion of Theorem \ref{thcrit} is a lower bound for $\dim
   F$ (which can be stated as a lower bound for the rank of a family
   of $m$ vectors in $\R^n$ seen as a $\Q$-vector space, see
   \cite{SFnestsev}, \S 2.3, Lemma 1). It is a natural question to ask
   whether this bound can be improved; we give a negative answer in
   Theorem \ref{thopti}. In the case of Nesterenko's linear
   independence criterion with only one point, Chantanasiri has given
   (\cite{Chantanasiri2}, \S 3) a very specific example of a point
   $\bfe_1= (\xi_1,\ldots,\xi_m)$ for which this bound is optimal
   (namely when $(\xi_1,\ldots,\xi_m)$ is a $\Q$-basis of a real
   number field of degree $m$). On the contrary, our result deals with
   generic tuples; it encompasses also Siegel's criterion, and the
   case of several points.

\begin{thm}
  \label{thopti}
  Let $m > n > 0$, and $F$ be a subspace of $\R^m$ defined over the
  rationals. Let $\tau_1,\ldots,\tau_n$, $\beta_1,\ldots,\beta_n$,
  $\eps$ be real numbers such that $\tau_1>0$, \ldots, $\tau_n>0$,
  $\eps>0$,
  \begin{equation} \label{eqhypopti} \tau_1+\ldots+\tau_n \leq \dim F
    - n \mbox{ and } \beta_1+\ldots+\beta_n = (1+\eps)(\dim F-1).
  \end{equation}
  Then for almost all $n$-tuples $(\bfe_1,\ldots,\bfe_n)\in F^n$ (with
  respect to Lebesgue measure) the following property holds.  For any
  sufficiently large integer $Q$ there exist $m$ linearly independent
  vectors $\bfqu, \ldots, \bfqm \in\Z^m$ such that, for any $ j \in
  \unm$:
  \begin{equation} \label{eqopti1} |\bfqj| \ll Q
  \end{equation}
  and
  \begin{equation} \label{eqopti2} Q^{-\tau_i}(\log Q)^{\beta_i -
      (1+\eps)\dim F} \ll | \bfqj \cdot \bfe_i | \ll Q^{-\tau_i}(\log
    Q)^{\beta_i } \mbox{ for any } i \in \unn,
  \end{equation}
  where the constants implied in the symbols $\ll$ depend on $m$, $n$,
  $F$, $\tau_1,\ldots,\tau_n$, $\beta_1,\ldots,\beta_n$, $\eps$,
  $\bfe_1,\ldots,\bfe_n$ but not on $Q$.
\end{thm}

This result will be proved in \S \ref{secpreuveopti}, using Theorem
\ref{thm:KG} and Minkowski's theorem on successive minima of a convex
body. We also postpone until \S \ref{secremopti} some remarks on
Theorem \ref{thopti}.

\bigskip

Throughout we will use the Vinogradov notation, \emph{i.e.}, for two
real quantities $x$ and $y$, we will write $x \ll y$ if there is a
constant $C> 0$ such that $x \leq Cy$. In Landau's $O$-notation this
would amount to writing $x = O(y)$. If $x \ll y$ and $y \ll x$, we
will write $x \asymp y$.


\section{A converse to linear independence criteria}

\subsection{Remarks  on Theorem \ref{thopti}} \label{secremopti}

We gather in this section several remarks on Theorem \ref{thopti}.

\begin{Rem} \label{remsuite} In general Nesterenko's criterion is
  stated under a slightly different assumption than $(ii)$ in Theorem
  \ref{thcrit}: it is assumed that there exist an increasing sequence
  $(Q_k)_{k\geq 1}$ of positive integers such that $Q_{k+1} =
  Q_k^{1+o(1)}$ as $k\to\infty$ (where the sequence denoted by $o(1)$
  tends to 0 as $k\to\infty$), and a sequence $(\bfq_k)_{k\geq 1}$ of
  vectors in $\Z^m$, such that for any $k$:
   $$| \bfq_k | \leq Q_k \mbox{ and } | \bfq_k  \cdot \bfe_i |  =  Q_k^{-\tau_i +o(1)}   \mbox{ for any } i \in \unn. $$
   Requesting also $\tau_1,\ldots,\tau_n$ to be pairwise distinct,
   this is actually equivalent to assumption $(ii)$ of Theorem
   \ref{thcrit}. In precise terms, if there is such a sequence $(Q_k)$
   then for any $Q$ sufficiently large one may choose the integer $k$
   such that $Q_k \leq Q < Q_{k+1}$, and let $\bfq = \bfq_k$. The
   converse is easy too: if assumption $(ii)$ of Theorem \ref{thcrit}
   holds, then one can choose {\em any} increasing sequence
   $(Q_k)_{k\geq 1}$ of positive integers such that $Q_{k+1} =
   Q_k^{1+o(1)}$ (for instance $Q_k = \beta^k$ with an arbitrary
   $\beta>1$) and let $\bfq_k$ be the vector corresponding to $Q =
   Q_k$.

   This remark shows that $ \tau_r(\underline \xi) =\tau'_r(\underline
   \xi) = \tau''_r(\underline \xi)$ for any $\underline \xi$ in the
   notation of \S 4.3 of \cite{eddzero}.  With the same notation,
   Theorem \ref{thopti} (with $F =\R^m$ and $n=1$) implies that this
   Diophantine exponent is equal to $m-1$ for almost all $\underline
   \xi = \bfe_1 \in \R^m$ (with respect to Lebesgue measure); this
   answers partly a question asked at the end of \cite{eddzero}.
\end{Rem}

\bigskip

\begin{Rem}
  In the setting of Theorem \ref{thopti}, if $\bfe_1,\ldots,\bfe_n$
  are $\Q$-linearly independent and belong to $F\cap\Qbar^m$ then
  applying Schmidt's Subspace Theorem instead of Theorem \ref{thm:KG}
  in the proof yields the same conclusion as that of Theorem
  \ref{thopti}, except that Eq. \eqref{eqopti2} is weakened to $ |
  \bfqj \cdot \bfe_i | = Q^{-\tau_i+o(1)}$.
\end{Rem}

\bigskip

In the rest of this section, we shall focus on the special case $m=2$,
$n=1$, $F=\R^2$. By homogeneity we may restrict to vectors $\bfe_1 =
(\xi,-1)$ with $\xi\in\R$. Since non-zero linear forms in $\xi$ and
$-1$ with integer coefficients are bounded from below in absolute
value if $\xi$ is a rational number, we assume $\xi$ to be irrational.
Recall that the irrationality exponent of $\xi$, denoted by
$\mu(\xi)$, is the supremum (possibly $+\infty$) of the set of $\mu>0$
such that there exist infinitely many $p,q\in\Z$ with $q>0$ such that
$|\xi-\frac{p}q|\leq q^{-\mu}$. Then the first question related to
Theorem \ref{thopti} is to know for which $\tau>0$ the following
holds:
\begin{eqnarray}
  &\mbox{For any $Q$ there exists $\bfq = (q_1,q_2)\in\Z^2\setminus\{(0,0)\}$ such that}& \label{eqdimun}  \\
  & |\bfq|\leq Q \mbox{  and } |q_1\xi-q_2|=Q^{-\tau+o(1)}.&\nonumber
\end{eqnarray}
Lemma 1 and Theorem 2 of \cite{eddzero} imply (using Remark
\ref{remsuite} above) that \eqref{eqdimun} holds if, and only if,
$\tau < \frac{1}{\mu(\xi)-1}$ (except maybe for $\tau =
\frac{1}{\mu(\xi)-1}$: this case is not settled in \cite{eddzero}).
This gives a satisfactory answer for any given $\xi$, and it would be
interesting to generalize it to arbitrary values of $m$ and $n$:
questions in this respect are asked (in the case $n=1$) in \S 4 of
\cite{eddzero}.  This result shows also that the conclusion of Theorem
\ref{thopti} does not hold for {\em any} $\bfe_1,\ldots,\bfe_n$:
property \eqref{eqdimun} fails to hold for $\tau=1$ if $\mu(\xi) > 2$.

\medskip

If $\xi$ is generic (with respect to Lebesgue measure), then $\mu(\xi)
= 2$ and the question left open in \cite{eddzero} is whether property
\eqref{eqdimun} holds for $\tau=1$.  Theorem \ref{thopti} answers this
question: it does, and the error term $Q^{o(1)}$ can be bounded
between powers of $\log Q$. Moreover, Theorem \ref{thopti} provides,
for any $Q$, two linearly independent vectors $\bfq$ as in
\eqref{eqdimun}: as far as we know, no result in the style of
\cite{eddzero} provides this conclusion for a non-generic $\xi$.

\medskip

In the same situation (namely with $m=2$, $n=1$, $F=\R^2$, and a
generic $\xi$), Theorem \ref{thopti} with $\tau_1 = 1 $ and
$\beta_1>1$ provides (for any $Q$) two linearly independent vectors
$\bfq = (q_1,q_2)\in\Z^2 $ such that $|\bfq|\ll Q$ and
\begin{equation} \label{equnlog}
Q^{-1} (\log Q)^{-\beta_1} \ll |  q_1\xi-q_2| \ll Q^{-1} (\log Q)^{ \beta_1} .
\end{equation}
The lower bound on $ | q_1\xi-q_2|$ is natural since for infinitely
many $Q$ there exists $\bfq $ such that $|\bfq|\leq Q$ and $Q^{-1}
(\log Q)^{-\beta_1} \ll | q_1\xi-q_2| \ll Q^{-1} (\log Q)^{ -1}$.  The
upper bound in Eq. \eqref{equnlog} could seem too large, since
Dirichlet's pigeonhole principle yields (for any $Q$) a non-zero $\bfq
$ such that $|\bfq|\leq Q$ and $ | q_1\xi-q_2| \ll Q^{-1} $. However
it is possible (by adapting the proof of Theorem \ref{thopti}) to
prove that, for infinitely many $Q$, all vectors $\bfq \in \Z^2 $ such
that $|\bfq| \ll Q$ and $ | q_1\xi-q_2| \ll Q^{-1} $ are collinear. To
obtain two linearly independent such vectors, one needs (for
infinitely many $Q$) to let $ | q_1\xi-q_2| $ increase a little more,
at least up to $Q^{-1} \log Q $: the upper bound in Eq.
\eqref{equnlog} is optimal (except that the case $\beta_1 = 1$ could
probably be considered, upon multiplying by a power of $\log \log Q$).

\subsection{Proof of Theorem \ref{thopti}} \label{secpreuveopti}

Before proving Theorem \ref{thopti}, let us outline the strategy in
the case where $F=\R^m$ and $ \tau_1+\ldots+\tau_n = \dim F - n $
(from which we shall deduce the general case). The convex body $\calC
\subset \R^m$ defined by \eqref{eqopti1} and the second inequality in
\eqref{eqopti2} has volume essentially equal to a power of $\log Q$.
There are non-zero integer points $\bfq$ inside $\calC$, but not ``too
far away inside'' (for $Q$ sufficiently large) : if $\bfq$ is such a
point and $\mu > 0$ is such that $\mu \bfq \in \calC$, then $\mu$ is
less than some power of $\log Q $ (otherwise the scalar products $|
\bfq \cdot \bfe_i|$ would be too small: this would contradict the
convergent case of Theorem \ref{thm:KG}). This is a lower bound on the
first successive minimum $\lambda_1$ of $\calC$. Using Minkowski's
convex body theorem, this yields an upper bound on the last successive
minimum $\lambda_m$, namely $\lambda_m \ll 1$. This concludes the
proof, except for the lower bound in Eq. \eqref{eqopti2} for which the
argument is similar: if $| \bfq ^{(j)} \cdot \bfe_i|$ is too small for
some $i,j$ then $(\bfe_1,\ldots,\bfe_n)$ is not generic (using again
the convergent case of Theorem \ref{thm:KG}).

\bigskip

Let us come now to a detailed proof of Theorem \ref{thopti}, starting
with the following remark.

\begin{Rem} \label{remegal} The general case of Theorem \ref{thopti}
  follows from the special case where the inequality in Eq.
  \eqref{eqhypopti} is an equality, that is $ \tau_1+\ldots+\tau_n =
  \dim F - n $. Indeed in general we have $ \tau_1+\ldots+\tau_n =
  \eta( \dim F - n ) $ with $0 < \eta \leq 1$, and applying the
  special case with $\tau_1/\eta$, \ldots, $\tau_n/\eta$ and $Q^\eta$
  yields the desired conclusion.
\end{Rem}

\bigskip

As a first step, let us assume that Theorem \ref{thopti} holds if
$F=\R^m$, and deduce the general case. Since $F$ is defined over $\Q$,
there exists a basis $(\bfu_1,\ldots,\bfu_d)$ of $F$ consisting in
vectors of $\Z^m$ (where $d = \dim F$; notice that Eq.
\eqref{eqhypopti} implies $d > n$). Let $\Phi:\R^d \to \R^m$ be the
linear map which sends the canonical basis of $\R^d$ to
$(\bfu_1,\ldots,\bfu_d)$. The special case of Theorem \ref{thopti}
applies to $\R^d$ (with the same parameters); it provides a subset
$\Ati \subset (\R^d)^n$ of full Lebesgue measure, and for any
$(\eti_1,\ldots,\eti_n)\in \Ati$ and any $Q$ sufficiently large $d$
linearly independent vectors $\bfqtiu, \ldots, \bfqtid\in\Z^d$.  Then
we let $A \subset F^n$ denote the set of all $n$-tuples
$(\bfe_1,\ldots,\bfe_n)$ given by $\bfe_1=\Phi(\eti_1)$, \ldots,
$\bfe_n = \Phi(\eti_n)$ with $(\eti_1,\ldots,\eti_n)\in \Ati$; this
subset $A$ has full Lebesgue measure in $F^n = (\Im \Phi)^n$.

Let us denote by $\Om \in M_d(\R) $ the matrix in the basis
$(\bfu_1,\ldots,\bfu_d)$ of the scalar product of $\R^m$ restricted to
$F$. This means that for any $\bfx, \bfy\in\R^d$ we have $\Phi(\bfx)
\cdot \Phi(\bfy) = \tra \bfx \Omega \bfy$, where $\bfx$ and $ \bfy$
are seen as column vectors (indeed they are the vectors of coordinates
in the basis $(\bfu_1,\ldots,\bfu_d)$ of $\Phi(\bfx) $ and $
\Phi(\bfy) $ respectively). This matrix $\Om$ has integer coefficients
(given by $\bfu_k \cdot \bfu_\ell$ for $1\leq k,\ell \leq d$), and a
non-zero determinant, so that $(\det \Om)\Om^{-1}$ is a matrix with
integer coefficients.

Let $(\bfe_1,\ldots,\bfe_n)\in A$, and  $Q$ be sufficiently large.  We let
$$\bfqj = \Phi\Big(  (\det \Om)\Om^{-1} \bfqtij\Big) \mbox{ for any } j \in \und,$$
so that
$$\bfqj \cdot \bfe_i = \tra \Big(  (\det \Om)\Om^{-1} \bfqtij\Big) \Om \eti_i =  (\det \Om) \bfqtij \cdot \eti_i  \mbox{ for any } i \in \unn$$
because $\Om$ is symmetric. Therefore Eqns. \eqref{eqopti1} and
\eqref{eqopti2} hold for $j \leq d$; moreover $\bfqu$, \ldots, $\bfqd$
are linearly independent vectors in $\Z \bfu_1+\ldots+\Z \bfu_d
\subset F \cap \Z^m$ (because the coefficients of $(\det \Om)\Om^{-1}$
are integers).

Since $F^\perp$ is a subspace of $\R^m$ defined over the rationals
(because $F$ is), there exists a basis $(\bfv_{d+1},\ldots, \bfv_m)$
of $F^\perp$ consisting in vectors of $\Z^m$. Then we let
$$\bfqdpu = \bfv_{d+1} + \bfqu, \ldots, \bfqm = \bfv_m + \bfqu.$$
Then $\bfqu, \ldots, \bfqm$ are linearly independent vectors in
$\Z^m$, and for any $j \in \{d+1,\ldots,m\}$ and any $i\in\unn$ we
have $\bfqj\cdot \bfe_i = \bfqu \cdot \bfe_i$ so that Eq.
\eqref{eqopti2} holds. Since $\bfv_{d+1},\ldots, \bfv_m$ can be chosen
independently from $Q$, we have also $|\bfqj| \ll |\bfqu| \ll Q$ so
that Eq. \eqref{eqopti1} holds too. This concludes the proof that the
full generality of Theorem \ref{thopti} follows from the special case
where $F=\R^m$.

\bigskip

From now on, we assume that $F=\R^m$ and prove Theorem \ref{thopti} in
this case.

Let $A_0$ denote the set of all $(\bfe_1,\ldots,\bfe_n)\in(\R^m)^n $
such that the system of inequalities
\begin{equation} \label{eqdefAz} | \bfq \cdot \bfe_i | \leq | \bfq |
  ^{-\tau_i} (\log|\bfq|) ^{\beta_i - (1+\eps)(1+\tau_i)} \mbox{ for
    any } i \in \unn
\end{equation}
holds for only finitely many $\bfq \in \Z^m$.

For any $i_0 \in \unn$, let $A_{i_0}$ denote the set of all
$(\bfe_1,\ldots,\bfe_n)\in(\R^m)^n $ such that the system of
inequalities
\begin{equation}
\left\{ \begin{array}{l}
| \bfq \cdot \bfe_i | \leq c_1  | \bfq | ^{-\tau_i} (\log|\bfq|) ^{\beta_i }    \mbox{ for any } i \in \unn, i \neq i_0 \\
| \bfq \cdot \bfe_{i_0} | \leq   | \bfq | ^{-\tau_{i_0}} (\log|\bfq|) ^{\beta_{i_0} - m (1+\eps)  }
\end{array} \right.
\end{equation}
holds for only finitely many $\bfq \in \Z^m$; here $c_1$ is a positive  constant the will be defined later in the proof (namely in Eq. \eqref{eqqj}), but could have be made explicit and stated here.

Using Eq. \eqref{eqhypopti} and Remark \ref{remegal}, the convergent
case of Theorem \ref{thm:KG} (with $(x_{1i},\ldots,x_{mi}) = \bfe_i$)
implies that $A_i \cap [-\frac12,\frac12]^{nm}$ has full Lebesgue
measure for any $i \in \{0,\ldots,n\}$. Since $A_i$ is stable under
multiplication by scalars, we have $A_i = \cup_{n \in \N} n (A_i \cap
[-\frac12,\frac12]^{nm})$ so that $A_i$ has full Lebesgue measure. At
last, let $\Ainf$ denote the set of all
$(\bfe_1,\ldots,\bfe_n)\in(\R^m)^n $ such that $\bfq \cdot \bfe_i \neq
0$ for any $\bfq\in\Z^m\setminus\{0\}$ and any $i \in\unn$. Then we
let $A = A_0 \cap A_1\cap\ldots\cap A_n \cap\Ainf$, and $A$ has full
Lebesgue measure in $(\R^m)^n $.

\bigskip

Let $(\bfe_1,\ldots,\bfe_n)\in A$, and $Q$ be sufficiently large. Let
$\calC$ denote the set of all $\bfq \in\R^m$ such that
\begin{equation} \label{eqdefC}
| \bfq | \leq Q \mbox{ and } | \bfq \cdot \bfe_i | \leq Q^{-\tau_i} (\log Q)^{\beta_i} \mbox{ for any } i \in \unn.
\end{equation}
Then $\calC$ is convex, compact, and symmetric with respect to the
origin. Its volume (denoted by $\vol(\calC)$) is such that
$\vol(\calC) \asymp (\log Q)^{(1+\eps)(m-1)}$, using both equalities
of Eq. \eqref{eqhypopti} (thanks to Remark \ref{remegal}) with $\dim F
= m$.

For any $j \in \unm$ let $\lam_j$ denote the infimum of the set of all
positive real numbers $\lam $ such that $\Z^m \cap \lam \calC$
contains $j$ linearly independent vectors, where $\lam \calC = \{ \lam
\bfq, \, \bfq \in \calC\}$. These $\lam_j$ are the successive minima
of the convex body $\calC$ with respect to the lattice $\Z^m$;
Minkowski's theorem (see for instance \cite{Cassels}, Chapter VIII)
yields $\frac{2^m}{m!} \leq \lam_1\ldots\lam_m \vol(\calC) \leq 2^m$,
so that
\begin{equation} \label{eqprodlam}
   \lam_1\ldots\lam_m   \asymp (\log Q)^{-(1+\eps)(m-1)}.
\end{equation}

Since $(\bfe_1,\ldots,\bfe_n)\in A_0$, for any $\bfq \in \Z^m
\setminus\{\bfze\}$ there exists $i \in\unn$ (which depends on
$\bfe_1,\ldots,\bfe_n$ and $\bfq$) such that
\begin{equation} \label{equtileAz}
| \bfq \cdot \bfe_i | \gg  | \bfq | ^{-\tau_i} (\log|\bfq|) ^{\beta_i - (1+\eps)(1+\tau_i)}
\end{equation}
where the constant implied in the symbol $\gg$ is small enough to take
into account the finitely many $\bfq \in \Z^m \setminus\{\bfze\} $
that satisfy Eq. \eqref{eqdefAz}; we have used here that $ \bfq \cdot
\bfe_i \neq 0$ for any $\bfq \in \Z^m \setminus\{\bfze\} $ and any
$i$, because $(\bfe_1,\ldots,\bfe_n)\in \Ainf$.

Let us deduce from this property that $ \lam_1 \gg (\log
Q)^{-(1+\eps)}$. With this aim in view, we let $\lam > 0$ be such that
$Q^{-1/2} \leq \lam \leq 1$ and $\lam \calC \cap \Z^m \neq \{\bfze\}$;
we are going to prove that $ \lam \gg (\log Q)^{-(1+\eps)}$. There
exists $\bfq' \in \calC$ such that $\bfq = \lam \bfq' \in \Z^m$ and
$\bfq \neq \bfze$. Then Eq. \eqref{equtileAz} provides an integer $i
\in\unn$ such that, using Eq. \eqref{eqdefC}:
$$  | \bfq | ^{-\tau_i} (\log|\bfq|) ^{\beta_i - (1+\eps)(1+\tau_i) } \ll   | \bfq \cdot \bfe_i | =  \lam   | \bfq ' \cdot \bfe_i | \leq \lam Q ^{-\tau_i}  (\log Q)^{\beta_i}.$$
Since we have also $| \bfq | = \lam | \bfq ' | \leq \lam Q$ and
$Q^{-1/2} \leq \lam \leq 1$ (so that $\log (\lam Q) \gg \log Q$), this
yields
$$\lam^{-\tau_i}Q ^{-\tau_i}   (\log Q)^{\beta_i- (1+\eps)(1+\tau_i) }  \ll   ( \lam Q)  ^{-\tau_i} (\log ( \lam Q) ) ^{\beta_i - (1+\eps)(1+\tau_i) } \ll  \lam Q ^{-\tau_i}  (\log Q)^{\beta_i},$$
thereby proving that $ \lam \gg (\log Q)^{-(1+\eps)}$. This concludes
the proof that $ \lam_1 \gg (\log Q)^{-(1+\eps)}$; since
$\lam_1\leq\ldots\leq\lam_m$ by definition of the successive minima,
this implies $\lam_j \gg (\log Q)^{-(1+\eps)}$ for any $j \in \unm$.
Plugging this lower bound for $j \leq m-1$ into Eq. \eqref{eqprodlam}
yields $\lam_m \ll 1$: there exist linearly independent vectors
$\bfqu, \ldots, \bfqm\in\Z^m$ such that
\begin{equation} \label{eqqj}
| \bfqj | \ll Q \mbox{ and } | \bfqj \cdot \bfe_i | \ll Q^{-\tau_i} (\log Q)^{\beta_i} \mbox{ for any } i \in \unn.
\end{equation}
This concludes the proof of Eq. \eqref{eqopti1},  and  that of the upper bound in Eq. \eqref{eqopti2}.

To prove  the lower bound in Eq. \eqref{eqopti2}, we start by noticing that Eq. \eqref{eqqj} yields
\begin{equation} \label{eqqjbis}
 | \bfqj \cdot \bfe_i | \leq c_1 Q^{-\tau_i} (\log Q)^{\beta_i} \mbox{ for any } i \in \unn
\end{equation}
for some positive constant $c_1$ (which could be made explicit); this
constant is the one used in the definition of $A_{i_0}$ at the
beginning of the proof.  Now let $i_0 \in \unn$.  Since
$(\bfe_1,\ldots,\bfe_n)\in A_{i_0}$ and $ \bfq \cdot \bfe_{i_0} \neq
0$ for any $\bfq \in \Z^m \setminus\{\bfze\} $, there exists a
positive constant $c_2$ such that no non-zero $\bfq\in\Z^m$ satisfies
the system of inequalities
\begin{equation}  \label{eqAizd}
\left\{ \begin{array}{l}
| \bfq \cdot \bfe_i | \leq c_1  | \bfq | ^{-\tau_i} (\log|\bfq|) ^{\beta_i }    \mbox{ for any } i \in \unn, i \neq i_0 \\
| \bfq \cdot \bfe_{i_0} | \leq c_2  | \bfq | ^{-\tau_{i_0}} (\log|\bfq|) ^{\beta_{i_0} - m (1+\eps)  }
\end{array} \right.
\end{equation}
For any $j \in \unm$, the non-zero vector $\bfqj$ satisfies the first
family of inequalities in \eqref{eqAizd} (thanks to Eq.
\eqref{eqqjbis}), so that
$$| \bfqj \cdot \bfe_{i_0} | >  c_2  | \bfqj | ^{-\tau_{i_0}} (\log|\bfqj|) ^{\beta_{i_0} - m (1+\eps)  }  \gg Q ^{-\tau_{i_0}} (\log Q |) ^{\beta_{i_0} - m (1+\eps)  }
$$
since $|\bfqj| \ll Q$. This concludes the proof of  the lower bound in Eq. \eqref{eqopti2}, and that of Theorem \ref{thopti}.

\section{Proof of Theorem \ref{thm:KG}}
\label{sec:proof-theorem}

\subsection{Convergence case for any choice of $m$ and $n$}
\label{sec:convergence-case}

In order to prove the convergence case, we will exhibit a family of
covers of $W_0(m,n,\underline{\psi})$. The covers will be the natural
ones, \emph{i.e.}, the cover of $W_0(m,n,\underline{\psi})$ by the
sets of solutions to \eqref{eq:2} for each individual $\mathbf{q}$.
Denote for each $\mathbf{q} \in \mathbb{Z}^m$, the set of matrices
with entries in $[-1/2, 1/2]$ satisfying the system of inequalities
\eqref{eq:2} by $B(\mathbf{q}, \underline{\psi})$. It is
straightforward to see that
\begin{equation}
  \label{eq:5}
  \lambda_{mn}(B(\mathbf{q}, \underline{\psi})) \asymp
  \psi_1(\abs{\mathbf{q}}) \cdots \psi_n(\abs{\mathbf{q}})
  \abs{\mathbf{q}}^{-n}.
\end{equation}
Here, the implied constants depend on $m$ and $n$.

Secondly, we will need to estimate the number of $\mathbf{q} \in
\mathbb{Z}^m \setminus \{\mathbf 0\}$ of a given norm, $r$ say. This is
however easily seen to be at most $(2m-1) r^{m-1}$, and so comparable
with $r^{m-1}$.

We now estimate the Lebesgue measure of $W_0(m,n,\underline{\psi})$
under the assumption of convergence. For each $N \geq 1$,
\begin{multline*}
  \lambda_{mn}\left(W_0(m,n,\underline{\psi})\right) \leq \lambda_{mn}\left(\bigcup_{r \geq N}
    \bigcup_{\abs{\mathbf{q}} = r} B(\mathbf{q}, \underline{\psi})\right)
  \leq \sum_{r \geq N} \sum_{\abs{\mathbf{q}} = r} \lambda_{mn}\left(B(\mathbf{q},
    \underline{\psi})\right) \\
  \leq \sum_{r \geq N} \sum_{\abs{\mathbf{q}} = r} \psi_1(r) \cdots
  \psi_n(r) r^{-n} \ll \sum_{r \geq N} \psi_1(r) \cdots \psi_n(r)
  r^{m-n-1}.
\end{multline*}
We have used \eqref{eq:5} and the counting estimates. The final sum is
the tail of a convergent series, which tends to zero as $N$ tends to
infinity.

\subsection{Divergence case}
\label{sec:divergence}
We give a general approach to the problem in question which has been adapted from the one used in \cite{hussain:_metric}.   In the case
$(m,n) \neq (2,1)$, we will not need the assumption of monotonicity of
the approximating functions. This will be clear from the proof below.

For each $\mathbf{q} \in \mathbb{Z}^{m-n}$, let
\begin{equation*}
  B_\mathbf{q} = \bigcup_{\substack{\mathbf{p} \in \mathbb{Z}^n \\
      \abs{\mathbf{p}} \leq \abs{\mathbf{q}}}} \left\{A \in M_{m \times n}
    ([-1/2,1/2]) : \abs{(\mathbf{p}, \mathbf{q})A}_i \leq
    \psi_i(\abs{\mathbf{q}}) \right\}
\end{equation*}
Writing each $A \in M_{m \times n} ([-1/2,1/2])$ as
$\binom{I_n}{\tilde{A}} X$, where $X$ is the $n \times n$ matrix
formed by the first $n$ rows of $A$, we find the related set
\begin{equation*}
  B'_\mathbf{q}(X) = \bigcup_{\substack{\mathbf{p} \in \mathbb{Z}^n \\
      \abs{\mathbf{p}} \leq \abs{\mathbf{q}}}}\left\{\tilde{A} \in
    M_{(m-n) \times n} ([-1/2,1/2]) : \abs{pX + \mathbf{q} \tilde{A} X}_i
    \leq  \psi_i(\abs{\mathbf{q}}) \right\}.
\end{equation*}
Finally, set $B'_\mathbf{q} = B'_\mathbf{q}(I_n)$.

Let $\epsilon > 0$ be fixed and sufficiently small. We will be more explicit later. From now on, we restrict ourselves to
considering matrices $A$ for which the determinant of the matrix $X$
consisting of the first $n$ rows of $A$ is $>\epsilon$. Evidently,
this determinant is also $\leq n!$. This immediately implies that $X$
is invertible with $(n!)^{-1} \leq \abs{\det(X^{-1})} < \epsilon^{-1}$.

\begin{lem}
  \label{lem:lem1}
  For each $X \in M_{n}([-1/2,1/2])$ with $\abs{\det(X)} > \epsilon$,
  and each $\mathbf{q}, \mathbf{q}_1, \mathbf{q}_2 \in
  \mathbb{Z}^{m-n}$,
  \begin{equation*}
    \lambda_{(m-n)n}(B'_\mathbf{q}(X)) \asymp_\epsilon
    \lambda_{(m-n)n}(B'_\mathbf{q}),
  \end{equation*}
  and
  \begin{equation*}
    \lambda_{(m-n)n}(B'_{\mathbf{q}_1}(X) \cap B'_{\mathbf{q}_2}(X))
    \asymp_\epsilon \lambda_{(m-n)n}( B'_{\mathbf{q}_1} \cap
    B'_{\mathbf{q}_2}) .
  \end{equation*}
\end{lem}

\begin{proof}
  Consider the defining inequalities for each set on the left hand
  sides. Multiplying by $X^{-1}$, we obtain a new system of
  inequalities, so that
  \begin{equation*}
    2^n \epsilon \prod_i \psi_i(\abs{\mathbf{q}}) \leq
    \lambda_{(m-n)n}(B'_\mathbf{q}(X)) \leq 2^n \epsilon^{-1} \prod_i
    \psi_i(\abs{\mathbf{q}}).
  \end{equation*}
Considering the special case
  when $X = I_n$, we obtain the first statement.

  The second statement is derived similarly, namely by considering the defining inequalities and multiplying by $X^{-1}$ to get an estimate for the measure.
\end{proof}

\begin{lem}
  \label{lem:2}
  For each pair $\mathbf{q}, \mathbf{q}'$,
  \begin{equation}
    \label{eq:6}
    \lambda_{mn}(B_\mathbf{q}) \asymp_\epsilon
    \lambda_{(m-n)n}(B'_\mathbf{q}),
  \end{equation}
  and
  \begin{equation}
    \label{eq:7}
    \lambda_{mn}(B_\mathbf{q} \cap B_{\mathbf{q}'})
    \asymp_{\epsilon} \lambda_{(m-n)n}(B'_\mathbf{q} \cap
    B'_{\mathbf{q}'}).
  \end{equation}
\end{lem}

\begin{proof}
  This follows on integrating out the $X$ and applying Lemma
  \ref{lem:lem1}. Indeed,
  \begin{equation*}
    \lambda_{mn}(B_\mathbf{q}) \asymp_\epsilon \int_{\substack{X \in M_n([-1/2,1/2])
        \\ \epsilon < \abs{\det(X)}}} \int_{\tilde{A} \in M_{(m-n)\times
        n}([-1/2,1/2])X^{-1}}
    \mathbf{1}_{B_\mathbf{q}}\left(\binom{I_n}{\tilde{A}}X\right)
    d\tilde{A} dX,
  \end{equation*}
  where $\mathbf{1}_{B_{\mathbf{q}}}$ denotes the characteristic
  function of $B_{\mathbf{q}}$. Let us prove that the inner integral is
  $\asymp_\epsilon \lambda_{(m-n)n}(B'_q(X))$.

  First, we deal with the case when $m-n>1$. For simplicity, we
  consider first the case $m=3, n = 1$ and extend subsequently. We are
  integrating over the set $M_{2\times 1}([-1/2,1/2])X^{-1}$, which is
  a square of area between $1$ and $\epsilon^{-2}$, since $X$ in this
  case is just a number between $\epsilon$ and $1$. Consider the intersection with each
  fundamental domain for the standard lattice $\mathbb{Z}^2$. Except for lower order terms arising at the boundary of  $M_{2\times 1}([-1/2,1/2])X^{-1}$, each such intersection will have measure $\lambda_{(m-n)n}(B'_{\mathbf{q}}(X))$.  The number of such
  contributing fundamental domains is bounded from below by $1$ and
  from above by $\epsilon^{-2}$. Hence, the result follows in this
  case.

  To get the full result for $m-n > 1$, the set $M_{(m-n)\times
    n}([-1/2,1/2])X^{-1}$ still covers at least $\frac{1}{n}
  M_{(m-n)\times n}([-1/2,1/2])$, as the entries of $X$ are between
  $-1/2$ and $1/2$. For $\abs{\bfq}$ large enough, the measure of the
  intersection of $B'_{\mathbf{q}}(X)$ with this set is $\asymp
  \frac{1}{n^{(m-n)n}} \lambda_{(m-n)n}(B'_{\mathbf{q}}(X))$,  and the result
  follows. The upper bound again follows as the determinant of $X$ is
  bounded from below.

  When $m-n=1$, the set consists of neighbourhoods of single points,
  and we simply count the contributions as usual. We have now shown that  the inner integral is
  $\asymp_\epsilon \lambda_{(m-n)n}(B'_q(X))$.

  To conclude, we use Lemma \ref{lem:lem1},
  \begin{multline*}
    \lambda_{mn}(B_\mathbf{q}) \asymp_\epsilon \int_{\substack{X \in
        M_n([-1/2,1/2]) \\ \epsilon < \abs{\det(X)}}}
    \lambda_{(m-n)n}(B'_q(X)) dX\\
    \asymp_\epsilon \int_{\substack{X \in M_n([-1/2,1/2]) \\ \epsilon
        < \abs{\det(X)}}} \lambda_{(m-n)n}(B'_\mathbf{q}) dX \asymp_\epsilon
    \lambda_{(m-n)n}(B'_\mathbf{q}).
  \end{multline*}
  The case of intersections follows similarly, this time using the
  second equation of Lemma \ref{lem:lem1}.
\end{proof}

At this point, proving the divergence case of the theorem is a
relatively straightforward matter. Indeed, a form of the divergence case of the Borel--Cantelli lemma
states that if $(A_n)$ is a sequence of sets in a probability space
with probability measure $\mu$ such that $\sum \mu(A_n) = \infty$,
then
\begin{equation}
  \label{eq:3}
  \mu\left(\bigcap_{k=1}^\infty \bigcup_{n = k}^\infty A_n\right) \geq
  \limsup_{N \rightarrow \infty} \frac{\left(\sum_{n=1}^N
      \mu(A_n)\right)^2}{\sum_{m,n = 1}^N \mu(A_m \cap A_n)}.
\end{equation}
If one can prove that for a sufficiently large set of pairs $(A_n, A_m)$, the denominator on the right hand side is $\ll \mu(A_m)\mu(A_n)$ whenever $m \neq n$, it follows from \eqref{eq:3} that the measure of the set on the left hand side is strictly positive. Even if this does not hold, one could hope for it to be true on average, so that the resulting right hand side would be positive. This is a standard technique in metric Diophantine approximation, with the property on the sets $A_n$ being called quasi-independence or in the latter case quasi-independence on average. It follows from Lemma \ref{lem:2}, that if a classical Khintchine--Groshev theorem can be established
using quasi-independence on average, then the measure of the absolute
value set is positive under the appropriate divergence assumption.

In the classical setup, one usually proves Khintchine--Groshev type
results using a variant of this lemma. Here, one applies the lemma
with  some subset of the family $B_{\mathbf{q}}'$ in place of $A_n$. In the simplest case, when $n=1$ and $m=3$, the family can be chosen to be those $\mathbf{q} = (p, \tilde{\mathbf{q}}) \in \mathbb{Z} \times \mathbb{Z}^{2}$ with the entries of $\tilde{\mathbf{q}}$ co-prime and the last entry positive. This will ensure that the corresponding sets $\cup_p B_{(p,\tilde{\mathbf{q}})}'$  are stochastically independent and hence quasi-independent. The fact that we take a union over $p$'s s critical. This gives a pleasing description of the sets involved as neighbourhoods of geodesics winding around a torus, and provides a simple argument for the stochastic independence of the sets. For details on this case, see \cite{MR1258756}. In that paper, the case $m-n>1$ is fully described. For the case when $m-n=1$, more delicate arguments are required. Below, we give references to work, where the refining procedure is carried out in each individual case.

For our purposes, in order to prove Theorem \ref{thm:KG}, using Lemma \ref{lem:2} we will
translate the right hand side of inequality \eqref{eq:3} to a
statement on the `classical' sets $B'_\mathbf{q}$ with the
corresponding limsup set.  In the case $m-n>2$, the required upper
bound on the intersections on average was established in
\cite{MR0159802} without the monotonicity assumption.  For $m-n = 2$,
the bound is found in \cite{hussain:_khint} and $m-n=1$, this is the
result of \cite{MR0157939}. In the last case, the monotonicity is
critical in the case $m=2, n=1$, as otherwise we could exploit the
Duffin--Schaeffer counterexample \cite{MR0004859} to arrive at a
counterexample to the present statement.

Having established that the measure is positive, it remains to prove
that the measure is full. To accomplish this, we apply an inflation
argument due to Cassels \cite{Casselsdio}, but tweaked to the absolute value setup. We pick a slowly decreasing
function $\tau(r)$ which tends to $0$, such that the functions $\psi'_i(r) =
\tau(r)\psi_i(r)$ satisfy the divergence assumption of the theorem.

One can show that the origin $0 \in \mat_{mn}(\mathbb{R})$ is a point of metric density for the set $W_0(m,n;\underline{\psi})$. This uses two properties. One is the fact that $0$ is an inner point of each set of matrices satisfying \eqref{eq:2} for a fixed $\bfq$. The other is the fact that the error function depends only on $\abs{\bfq}$, so the parallelepiped of matrices satisfying \eqref{eq:2} does not change shape but only orientation as $\bfq$ varies over integer vectors with the same height $\abs{\bfq}$. Since the distribution of angles of integer vectors of the same height becomes uniform as the height increases, this implies that the origin must be a point of metric density.

Now, by the Lebesgue Density Theorem, for almost every matrix $A \in \R^{mn}$, there is a matrix near the origin $\tilde{A} \in W_0(m,n;\underline{\psi}')$ and a real number $r$, such that $A = r \tilde{A}$. That is,
\begin{equation*}
\abs{\bfq A}_i = \abs{\bfq r\tilde{A}}_i = \abs{r}\abs{\bfq \tilde{A}}_i < r \psi'_i(\abs{\bfq}),
\end{equation*}
for infinitely many $\bfq$.
This implies that $A \in W_0(m,n;\underline{\psi})$, since $r$ is fixed and $\tau$ tends to 0.

 \newcommand{\url}{\texttt}


\begin{thebibliography}{10}

\bibitem{BR}
K.M. Ball and T.~Rivoal.
\newblock Irrationalit{\'e} d'une infinit{\'e} de valeurs de la fonction z\^eta
  aux entiers impairs.
\newblock {\em Invent. Math.}, 146(1):193--207, 2001.

\bibitem{Boualg}
N.~Bourbaki.
\newblock {\em Alg\`ebre}, chapter~II.
\newblock Hermann, third. edition, 1962.

\bibitem{Casselsdio}
J.W.S. Cassels.
\newblock {\em An introduction to {D}iophantine approximation}.
\newblock Number~45 in Cambridge Tracts in Math. and Math. Phys. Cambridge
  University Press, 1957.

\bibitem{Cassels}
J.W.S. Cassels.
\newblock {\em An introduction to the geometry of numbers}.
\newblock Number~99 in Grundlehren der Math. Wiss. Springer, 1959.

\bibitem{Chantanasiri2}
Amarisa Chantanasiri.
\newblock G\'en\'eralisation des crit\`eres pour l'ind\'ependance lin\'eaire de
  {N}esterenko, {A}moroso, {C}olmez, {F}ischler et {Z}udilin.
\newblock {\em Ann. Math. Blaise Pascal}, 19(1):75--105, 2012.

\bibitem{dickinson}
D.~Dickinson and M.~Hussain.
\newblock The metric theory of mixed type linear forms.
\newblock {\em Int. J. Number Theory}, 9(1):77--90, 2013.

\bibitem{MR1258756}
M.~M. Dodson.
\newblock Geometric and probabilistic ideas in the metric theory of
  {D}iophantine approximations.
\newblock {\em Uspekhi Mat. Nauk}, 48(5(293)):77--106, 1993.

\bibitem{MR0004859}
R.~J. Duffin and A.~C. Schaeffer.
\newblock Khintchine's problem in metric {D}iophantine approximation.
\newblock {\em Duke Math. J.}, 8:243--255, 1941.

\bibitem{EMS}
N.I. Fel'dman and Yu.V. Nesterenko.
\newblock {\em Number Theory {IV}, Transcendental Numbers}.
\newblock Number~44 in Encyclopaedia of Mathematical Sciences. Springer, 1998.
\newblock A.N. Parshin and I.R. Shafarevich, eds.

\bibitem{SFnestsev}
S.~Fischler.
\newblock Nesterenko's linear independence criterion for vectors.
\newblock Preprint arxiv 1202.2279 [math.NT], submitted, 2012.

\bibitem{eddzero}
S.~Fischler and T.~Rivoal.
\newblock Irrationality exponent and rational approximations with prescribed
  growth.
\newblock {\em Proc. Amer. Math. Soc.}, 138(8):799--808, 2010.

\bibitem{MR0157939}
P.~Gallagher.
\newblock Metric simultaneous diophantine approximation.
\newblock {\em J. London Math. Soc.}, 37:387--390, 1962.

\bibitem{groshev38}
A.~V. Groshev.
\newblock A theorem on systems of linear forms.
\newblock {\em Dokl. Akad. Nauk SSSR}, 19:151--152, 1938.

\bibitem{hussain:_metric}
M.~Hussain and S.~Kristensen.
\newblock Metrical results on systems of small linear forms.
\newblock {\em Int. J. Number Theory} 9(3), 2013.

\bibitem{hussain}
M.~Hussain and J.~Levesley.
\newblock The metrical theory of simultaneously small linear forms.
\newblock To appear in {\em Funct. Approx. Comment. Math.}

\bibitem{hussain:_khint}
M.~Hussain and T.~Yusupova.
\newblock A note on the weighted {K}hintchine--{G}roshev theorem.
\newblock Preprint.

\bibitem{MR1544787}
A.~Khintchine.
\newblock Zur metrischen {T}heorie der diophantischen {A}pproximationen.
\newblock {\em Math. Z.}, 24(1):706--714, 1926.

\bibitem{Nesterenkocritere}
Yu.V. Nesterenko.
\newblock On the linear independence of numbers.
\newblock {\em Vestnik Moskov. Univ. Ser. I Mat. Mekh. [Moscow Univ. Math.
  Bull.]}, 40(1):46--49 [69--74], 1985.

\bibitem{RivoalCRAS}
T.~Rivoal.
\newblock La fonction z\^eta de {R}iemann prend une infinit{\'e} de valeurs
  irrationnelles aux entiers impairs.
\newblock {\em C. R. Acad. Sci. Paris, Ser. I}, 331(4):267--270, 2000.

\bibitem{MR0159802}
W.~M. Schmidt.
\newblock Metrical theorems on fractional parts of sequences.
\newblock {\em Trans. Amer. Math. Soc.}, 110:493--518, 1964.

\end{thebibliography}

\end{document}